\documentclass[reqno]{amsart}
\usepackage{amsmath,amssymb,amsfonts}
   \DeclareMathOperator{\Id}{Id}
   \DeclareMathOperator{\e}{e}
   
   \DeclareMathOperator{\Lip}{Lip}
\usepackage{dsfont}
   \newcommand{\N}{\mathbb{N}}
   \newcommand{\R}{\mathbb{R}}
   \newcommand{\Z}{\mathbb{Z}}
\usepackage{euscript}
   \def\cA{\EuScript{A}}
   
   \def\cF{\EuScript{F}}

   \def\cC{\EuScript{C}}
\usepackage{enumitem}
   \newcommand{\lb}{label}
   \newcommand{\lm}{leftmargin}

\newtheorem{theorem}{Theorem}
\newtheorem{corollary}[theorem]{Corollary}

\renewcommand{\epsilon}{\varepsilon}

\newcommand{\lbd}{\lambda}

\newcommand{\prts}[1]{\left(#1\right)}
\newcommand{\prtsr}[1]{\left[#1\right]}
\newcommand{\abs}[1]{\left|#1\right|}

\newcommand{\set}[1]{\left\{#1\right\}}
\newcommand{\setm}[1]{\setminus\set{#1}}
\newcommand{\sups}[1]{\sup\set{#1}}

\newcommand{\dsum}{\displaystyle\sum}

\newcommand{\ov}[1]{\overline{#1}}

\newcommand{\Es}{\hspace{0.5cm}}
\newcommand{\dst}{\displaystyle}
\renewcommand{\ge}{\geqslant}
\renewcommand{\geq}{\geqslant}
\renewcommand{\le}{\leqslant}
\renewcommand{\leq}{\leqslant}
\begin{document}
\title[Nonuniform behavior and stability of hopfield...]
   {Nonuniform behavior and stability of Hopfield neural networks with delay}
\author[Ant\'onio J. G. Bento]{Ant\'onio J. G. Bento}
\address{A. Bento\\
   Departamento de Matem\'atica\\
   Universidade da Beira Interior\\
   6201-001 Covilh\~a\\
   Portugal}
\email{bento@mat.ubi.pt  \protect\phantom{\cite{*}}}
\author{Jos\'e J. Oliveira}
\address{J. Oliveira\\
   Centro de Matem\'atica\\
   Universidade do Minho\\
   Campus de Gualtar\\
   4710-057 Braga\\
   Portugal}
\email{jjoliveira@math.uminho.pt}
\author{C\'esar M. Silva}
\address{C. Silva\\
   Departamento de Matem\'atica\\
   Universidade da Beira Interior\\
   6201-001 Covilh\~a\\
   Portugal}
\email{csilva@mat.ubi.pt}
\urladdr{www.mat.ubi.pt/~csilva}
\date{\today}
\begin{abstract}
   We obtain a result on the behavior of the solutions of a general nonautonomous Hopfield neural network model with delay, assuming some general bound for the product of consecutive terms in the sequence of neuron charging times and some conditions to control the nonlinear part of the equations. We then apply this result to improve some existent results on the stability of Hopfield models. Our results are based on a new abstract result on the behavior of nonautonomous delayed equations.
\end{abstract}
\maketitle
\section{Introduction}
Due to their many applications in various engineering and scientific areas such as signal processing, image processing and pattern classification (see \cite{Chua_Yang-IEEETCS-1988-2,Chua_Yang-IEEETCS-1988}), neural network models are still a subject of active research. Concerning neural networks, one of the most important objects of study is the obtention of conditions guarantying the global stability of equilibrium states \cite{Mohamad-PD-2001,Mohamad_Gopalsamy-MCS-2000,Mohamad_Gopalsamy-AMC-2003}, periodic solutions \cite{Huang_Mohamad_Xia-CMM-2009,Xu_Wu-ADE-2013} or, more generally, of a particular solution \cite{Huang_Wang_Gao-PLA-2006}.

In the present work we consider a discrete-time nonautonomous neural network with time delay. The relevance of our setting is easily clarified. Firstly, although theoretically speaking neural networks should be described by a continuous-time model, it is essential to formulate discrete-time versions that can be implemented computationally \cite{Mohamad_Gopalsamy-AMC-2003,Mohamad_Naim-JCAM-2002}. Secondly, it is important to consider delay in modelling neural networks in order to reproduce the effect of finite transmission speed of signals among neurons (there is a mathematical counterpart of this since time delay may cause instability and oscillation \cite{Marcus_Westervelt-PRA-1989}). At last, nonautonomy is associated to the change of parameters such as neuron charging time, interconnection weights and external inputs in the course of time. This nonautonomy can be translated not only by time-varying parameters, but also by time-varying delays \cite{Chen_Lu_Liang-PLA-2006, Liu_Tang_Martin_Liu-PLA-2007, Udpin_Niamsup-DDNS-2013, Yu_Zhang_Fei-NARWA-2010}. There are still few stability results in the context of nonautonomous nonperiodic neural network models~\cite{Wei_Zhou_Zhang-CMA-2009}.

Unlike the usual assumptions, our conditions correspond to bound the product of consecutive terms in the sequence of neuron charging times and to assume some condition controlling the nonlinear part of the equations. The classic method of proof is to consider that there is an equilibrium point or a periodic solution and to construct a suitable Lyapunov function to assure the global stability of this particular solution \cite{Chua_Yang-IEEETCS-1988, Gao_Cui-AMM-2009, Mohamad-PD-2001, Mohamad_Gopalsamy-AMC-2003,  Wei_Zhou_Zhang-CMA-2009}. The technique used here is different from the usual ones, namely we see our system as a sufficiently small perturbation of a nonuniform contraction and use Banach's fixed point theorem in some suitable complete metric space to obtain the global stability of our system. This approach allows us to dismiss the requirement of existence of a stationary or, more generally, periodic solution and additionally to consider a more general form for the nonlinear part of the model. When we restrict to the particular case of a periodic Hopfield model, our conditions for existence of a globally stable periodic solution generalize the existent results in the literature \cite{Xu_Wu-ADE-2013}.

This work is organized in the following way. In section \ref{section:applications} we use the discretization technique in~\cite{Mohamad_Gopalsamy-AMC-2003} to obtain a discrete version of a generalized neural network model that includes the well known Hopfield neural network models considered in~\cite{Mohamad_Gopalsamy-AMC-2003,Xu_Wu-ADE-2013} and the bidirectional associative memory neural network models studied in~\cite{Huang_Mohamad_Xia-CMM-2009,Mohamad-PD-2001}. Next, we state our main stability result and we see that it is a consequence of the abstract result considered in section \ref{section:stability...}. As a corollary, we get global exponential stability for the models in~\cite{Wei_Zhou_Zhang-CMA-2009} under distinct hypothesis from the ones assumed in that paper. After, for the periodic model, considering a Poincar\'e map, we obtain the existence of a periodic solution as a consequence of the global exponential stability. This result improves one of the main results in~\cite{Xu_Wu-ADE-2013}. Finally, in section \ref{section:stability...}, we consider general discrete-time delayed models that include our neural network models as particular cases and obtain an abstract global stability result that establishes in particular the stability results in section \ref{section:applications}.
\section{Hopfield Models}\label{section:applications}
As a generalization of the continuous-time Hopfield neural network models presented in \cite{Mohamad_Gopalsamy-AMC-2003, Xu_Wu-ADE-2013} we have
\begin{equation} \label{hopfield}
   x'_i(t)
   = -a_i(t)x_i(t)+\sum_{j=1}^N k_{ij}(t,x_j(t-\alpha_{ij}(t))),\Es t\geq 0,\,i=1,\ldots,N,
\end{equation}
where $a_i:[0,+\infty[ \to [0,+\infty[$, $k_{ij}:[0,+\infty[ \times \R \to \R$, and $\alpha_{ij}:[0,+\infty[\to[0,+\infty[$ are continuous functions with $\alpha_{ij}$ bounded and $k_{ij}$ Lipschitz on the second variable. Here $a_i(t)$ is the neuron charging time.

Following the ideas in \cite{Mohamad_Gopalsamy-AMC-2003}, to obtain a discrete-time analogue of the continuous-time model \eqref{hopfield}, we consider the following approximation
\begin{equation} \label{hopfield-2}
      x'_i(t)= -a_i([t/h]h)x_i\prts{t}
      + \sum_{j=1}^N k_{ij}\prts{[t/h]h,x_j\prts{[t/h]h- \prtsr{\dfrac{\alpha_{ij}\prts{[t/h]h}}{h}} h}},
\end{equation}
$i=1,\ldots,N$, $t\in[mh,(m+1)h[$ for $m\in\N_0$, where $h$ is a fixed positive real number (discretization step size) and $[r]$ denotes the integer part of the real number $r$. Clearly, for $t\in[mh,(m+1)h[$ we have $[t/h]=m$ and
the model \eqref{hopfield-2} has the form
\begin{equation*} 
  x'_i(t)=-a_i(mh)x_i(t)+ \sum_{j=1}^N k_{ij}\prts{mh,x_j\prts{\prts{m- \prtsr{\dfrac{\alpha_{ij}(mh)}{h}}}h}},
\end{equation*}
which is equivalent to
\begin{equation} \label{hopfield-4}
  e^{a_i(mh)t}x'_i(t)+a_i(mh)e^{a_i(mh)t}x_i(t)=e^{a_i(mh)t} \sum_{j=1}^N k_{ij}\prts{mh,x_j\prts{\prts{m- \tau_{ij}(m)}h}},
\end{equation}
where
   $$\tau_{ij}(m)
      = \prtsr{\dfrac{\alpha_{ij}(mh)}{h}}.$$
Integrating \eqref{hopfield-4} over $[mh,t[$, with $t<(m+1)h$, we obtain
\begin{equation*} 
  \int_{mh}^t\prtsr{e^{a_i(mh)s}x_i(s)}'ds= \left(\frac{e^{a_i(mh)t}-e^{a_i(mh)mh}}{a_i(mh)}\right)
  \sum_{j=1}^N k_{ij}\prts{mh,x_j\prts{\prts{m- \tau_{ij}(m)}h}},
\end{equation*}
which is equivalent to
\begin{equation*} 
  x_i(t)= e^{a_i(mh)(mh-t)}x_i(mh)+ \left(\frac{1-e^{a_i(mh)(mh-t)}}{a_i(mh)}\right) \sum_{j=1}^N k_{ij}\prts{mh,x_j\prts{\prts{m- \tau_{ij}(m)}h}}.
\end{equation*}

Letting $t\to(m+1)h$, we obtain
\begin{equation} \label{hopfield-7}
  x_i((m+1)h)=e^{-a_i(mh)h}x_i(mh)+\left(\frac{1-e^{-a_i(mh)h}}{a_i(mh)}\right) \sum_{j=1}^N k_{ij}\prts{mh,x_j\prts{\prts{m- \tau_{ij}(m)}h}}.
\end{equation}
Thus, identifying $x_i(mh)$ with $x_i(m)$, $a_i(mh)$ with $a_i(m)$ and $k_{ij}(mh,\,\cdot\,)$ with $k_{ij}(m,\,\cdot\,)$ and defining
\begin{equation} \label{def:theta_i(m)=...}
   \theta_i(m)=\frac{1-e^{-a_i(m)h}}{a_i(m)},
\end{equation}
   $$ $$
equation~\eqref{hopfield-7} becomes
\begin{equation} \label{hopfield-8}
  x_i(m+1)=e^{-a_i(m)h}x_i(m)+\theta_i(m) \sum_{j=1}^N k_{ij}\big(m,x_j(m-\tau_{ij}(m))\big).
\end{equation}

The model~\eqref{hopfield-8} can be rewritten in the following way
\begin{equation} \label{HOPFIELD}
    x_i(m+1)=c_i(m)x_i(m)+\dst\sum_{j=1}^N h_{ij}\big(m,x_j(m-\tau_{ij}(m))\big),
\end{equation}
$i=1,\ldots,N$, $m\in\N_0$, where $c_i:\N_0 \to ]0,1[ $, $\tau_{ij}:\N_0\to\N_0$ are bounded functions with $\tau:=\dst\max\{\tau_{ij}(m):m\in\N,i,j=1,\ldots,N\}$, and $h_{ij}: \N_0 \times \R \to \R$ are Lipschitz functions on the second variable, i.e., there exist $H_{ij}:\N_0\to\R^+$ such that
   $$ |h_{ij}(m,u)-h_{ij}(m,v)|
      \leq H_{ij}(m)|u-v|,\,\,\,\forall u,v\in\R,\,m\in \N_0.$$
In this paper we consider the Hopfield neural network model~\eqref{HOPFIELD} that generalizes some existent models in the literature~\cite{Huang_Mohamad_Xia-CMM-2009, Mohamad-PD-2001, Mohamad_Gopalsamy-AMC-2003, Xu_Wu-ADE-2013}.

Before stating our main result, we need to introduce some notation. Let
   $$ \Delta = \set{\prts{m,n} \in \Z^2 \colon m \ge n \ge 0}$$
and, given a set $I\subseteq\R$ and a number $r\in\Z^-$, define $I_\Z = I \cap \Z$. Consider the space $X$ of the functions
   $$ \alpha:[r,0]_\Z \to \R$$
equipped with the norm
   $$ \|\alpha\| = \max\limits_{j=r, \ldots,0} |\alpha(j)|.$$
Given $N \in \N$, we are going to consider the cartesian products $X^N$ and $\R^N$ equipped with the supremum norm, i.e., for $\ov\alpha = \prts{\alpha_1, \ldots, \alpha_N} \in X^N$ and $\ov y = \prts{y_1, \ldots, y_N} \in \R^N$, we have
   $$ \|\ov\alpha\|
      = \max_{i=1, \ldots,N} \|\alpha_i\|
      = \max_{i=1, \ldots,N} \prts{\max_{j=r, \ldots, 0} \abs{\alpha_i(j)}}$$
and
   $$ |\ov y| = \max_{i=1, \ldots, N} |y_i|.$$

Given $n \in \N_0$ and a function $\ov x \colon [n+r,+\infty[_\Z \to \R^N$ we denote the $i$th component by $x_i$, i.e., $\ov x = \prts{x_1, \ldots x_N}$. For each $m \in \N_0$ such that $m \ge n$, we define $\ov x_m \in X^N$ by
   $$ \ov x_m(j) = \ov x(m+j), \ j=r, r+1, \ldots, 0.$$

For each $n \in \N_0$ and each $\ov\alpha \in X^N$, we denote by $\ov x (\cdot,n,\ov\alpha)$ the unique solution
   $$ \ov x:[n+r,+\infty[_\Z \to \R^N$$
of~\eqref{HOPFIELD} with initial conditions $\ov x_n = \ov\alpha$.

We now state our main global stability result for the neural network model given by~\eqref{HOPFIELD}. This theorem furnishes a bound for the distance between solutions of~\eqref{HOPFIELD} based on bound for the products of consecutive neuron charging times, assuming that the Lipschitz constants of the nonlinear part of the model are sufficiently small. We will use it to obtain several results on the stability of several neural networks.

\begin{theorem}\label{thm:aplicacao-cont}
   Consider model~\eqref{HOPFIELD} and assume that there exist a double sequence $(a'_{m,n})_{(m,n) \in \Delta}$ such that
	\begin{equation} \label{eq:cond-aplicacao-cont-1}
      a^{(i)}_{m,n}
      := \prod_{s=n}^{m-1}c_i(s)\leq a'_{m,n},
	\end{equation}
   for all $i=1, \ldots, N$ and all $(m,n) \in \Delta$, and
	\begin{equation*} 
		\lambda
		:= \max_{i=1,\ldots,N} \prtsr{\sup_{(m,n) \in \Delta}
         \set{\frac{1}{a'_{m,n}} \sum_{k=n}^{m-1} a^{(i)}_{m,k+1} a'_{k,n}
            \, \dst \sum_{j=1}^N H_{ij}(k)}}
      < 1.
	\end{equation*}
   Then, for every $\ov\alpha,\ov\alpha^*:[r,0]_\Z\to\R^N$ and every $(m,n) \in \Delta$, we have
	\begin{equation*}
		\|\ov x_m(\cdot,n,\ov\alpha)-\ov x_m(\cdot,n,\ov\alpha^*)\|
		\le \dfrac{1}{1-\lbd} \, a'_{m,n} \|\ov\alpha-\ov\alpha^*\|.
	\end{equation*}
\end{theorem}

\begin{proof}
	Consider $n\in\N_0$ and $\ov\alpha, \ov\alpha^*:[r,0]_\Z\to\R^N$. The change
	$$ \ov y(m)=\ov x(m,n,\ov\alpha)-\ov x(m,n,\ov\alpha^*)$$
	transforms \eqref{HOPFIELD} into the system
	\begin{equation} \label{HOPFIELD-2}
		y_i(m+1)
		= c_i(m)y_i(m) +\dsum_{j=1}^N \widetilde{h}_{ij}(m,y_j(m-\tau_{ij}(m))),
	\end{equation}
	$i=1,\ldots,N$, $m \ge n$, where
	  $$ \widetilde{h}_{ij}(m,u)
	     = h_{ij}(m,u+x_j(m-\tau_{ij}(m),n,\ov \alpha^*))
         - h_{ij}(m,x_j(m-\tau_{ij}(m),n,\ov \alpha^*)).$$
   Now, $\ov y=0$ is an equilibrium point of \eqref{HOPFIELD-2} and, by Theorem~\ref{thm:abstract}, we obtain, for all function $\ov\beta:[r,0]_\Z\to\R^N$,
	  $$ \|\ov y_m(\cdot,n,\ov\beta)\|
	     \le \dfrac{1}{1-\lbd} \, a'_{m,n} \|\ov\beta\|$$
   for all $m \ge n$. Letting $\ov\beta=\ov\alpha-\ov\alpha^*$, we conclude that
	  $$ \|\ov x_m(\cdot,n,\ov\alpha) -\ov x_m(\cdot,n,\ov\alpha^*)\|
	     = \|\ov y_m(\cdot,n,\ov\alpha-\ov\alpha^*)\|
	     \le \dfrac{1}{1-\lbd} \, a'_{m,n} \|\ov\alpha-\ov\alpha^*\|$$
   for all $m \ge n$.
\end{proof}

In~\cite{Xu_Wu-ADE-2013} the authors considered the following discretization of a nonautonomous continuous-time Hopfield neural network model, which is a particular case of~\eqref{HOPFIELD},
\begin{equation} \label{HOPFIELD-Xu-Wu}
	x_i(m+1)
   = x_i(m)\e^{-a_i(m)h}+\theta_i(m)
      \left[ \dst\sum_{j=1}^N b_{ij}(m)f_j(x_j(m-\tau(m)))+I_i(m) \right],
\end{equation}
$i=1,\ldots,N$, where $a_i, b_{ij}, I_i, \tau : \N_0 \to \R$ are bounded functions with $a_i(m) > 0$, $0 \le \tau(m) \le \tau$, $f_j:\R \to \R$  are Lipschitz functions with Lipschitz constant $F_j>0$, $\theta_i(m)$ is given by~\eqref{def:theta_i(m)=...} and $h>0$ ($h$ is the discretization step size). We are going to use the following notation
   $$ a_i^-
      = \inf_m a_i(m) \ \ \ \text{ and } \ \ \
      b_{ij}^+ = \sup_m |b_{ij}(m)|
      \ \ \ \text{ and } \ \ \
      \theta_i^+ = \sup_m \theta_i (m).$$

We have the following result that establishes the global exponential stability of all solutions of~\eqref{HOPFIELD-Xu-Wu}.

\begin{corollary}\label{thm:aplicacao-cont-Xu-Wu}
	If
	\begin{equation} \label{eq:a_i^->alpha...}
		a_i^- > \sum_{j=1}^{N} b_{ij}^+ F_j.
	\end{equation}
   for every $i=1, \ldots, N$, then model~\eqref{HOPFIELD-Xu-Wu} is globally exponentially stable, i.e., there are constants $\mu>0$ and $C>1$ such that
	\begin{equation*}
		\|\ov x_m(\cdot,n,\ov\alpha)-\ov x_m(\cdot,n,\ov\alpha^*)\|
		\le  \,C \e^{-\mu(m-n)} \|\ov\alpha-\ov\alpha^*\|
	\end{equation*}
   for every $\ov\alpha,\ov\alpha^*:[-\tau,0]_\Z\to\R^N$ and every $(m,n) \in \Delta$.
\end{corollary}

\begin{proof}
   We will show that we are in the conditions of Theorem~\ref{thm:aplicacao-cont}. Defining $\nu_i=a_i^-h$, by ~\eqref{eq:a_i^->alpha...} there is positive number $\mu < \min\limits_i\nu_i$ such that
	\begin{equation} \label{eq:a_i^->alpha+beta...}
		\dfrac{ \e^{\nu_i-\mu}-1}{\e^{\nu_i} -1}a_i^- >
		\sum_{j=1}^{N} b_{ij}^+ F_j,\ \ \ \forall i=1,\ldots,N.
	\end{equation}
   Putting $a'_{m,n}=\e^{-\mu (m-n)}$, condition~\eqref{eq:cond-aplicacao-cont-1} is trivially satisfied because
      $$ a^{(i)}_{m,n}
         \le \e^{-\nu_i (m-n)}
         \le \e^{-\mu (m-n)}.$$
   Since $\theta_i^+ = \dfrac{1-\e^{-\nu_i}}{a_i^-}$, we have by~\eqref{eq:a_i^->alpha+beta...} that
	\begin{align*}
		\lbd
		& = \max_{i=1,\ldots,N}\prtsr{\sup_{(m,n) \in \Delta}
         \set{\dfrac{1}{a'_{m,n}}
			\sum_{k=n}^{m-1} a^i_{m,k+1}a'_{k,n} \, \theta_i(k)
         \sum_{j=1}^N |b_{ij}(k)| F_j}} \\
		& \leq \max_{i=1,\ldots,N}\prtsr{\sup_{(m,n) \in \Delta}
         \set{\e^{\mu (m-n)}
			\sum_{k=n}^{m-1} \e^{-\nu_i (m-k-1)-\mu (k-n)}} \theta^+_i
         \sum_{j=1}^N b_{ij}^+ F_j}\\
		& < \max_{i=1,\ldots,N}\prtsr{\sup_{(m,n) \in \Delta}
         \set{\sum_{k=n}^{m-1} \e^{(\nu_i-\mu)(k-m)}} \ \e^{\nu_i} \
			\dfrac{1-\e^{-\nu_i}}{a_i^-}
         \dfrac{\e^{\nu_i-\mu}-1}{\e^{\nu_i} -1}a_i^-}\\
		& = \max_{i=1,\ldots,N}\prtsr{\sup_{(m,n) \in \Delta}
         \set{\dfrac{1-\e^{(\nu_i-\mu)(n-m)}}{\e^{\nu_i-\mu}-1}} \
			\dfrac{(\e^{\nu_i}-1)(\e^{\nu_i-\mu}-1)}{(\e^{\nu_i}-1)a^-_i}a_i^-}\\
		& = \max_{i=1,\ldots,N} \prtsr{\sup_{(m,n) \in \Delta}
         \set{1-\e^{(\nu_i-\mu)(n-m)}}}\\
		&  = 1
	\end{align*}
	and this proves the corollary.
\end{proof}

In the next corollary we slightly improve condition~\eqref{eq:a_i^->alpha...} in the last corollary. To do that we need to define the concept of an $M$-matrix. We say that a square real matrix is an $M$-matrix if the off-diagonal entries are nonpositive and all the eigenvalues have positive real part.

Now consider the $N\times N$-matrix $\mathcal{M}$ defined by
\begin{equation*} 
	\mathcal{M}=diag(a_1^-,\ldots,a_N^-)-\left[b^+_{ij}F_j\right]
\end{equation*}

\begin{corollary}\label{thm:aplicacao-cont-Xu-Wu-2}
   If $\mathcal{M}$ is an M-matrix, then the model \eqref{HOPFIELD-Xu-Wu} is global exponential stable, i.e., there are $\mu>0$ and $C>1$ such that
	\begin{equation*}
		\|\ov x_m(\cdot,n,\ov\alpha)-\ov x_m(\cdot,n,\ov\alpha^*)\|
		\le  \,C \e^{-\mu(m-n)} \|\ov\alpha-\ov \alpha^*\|.
	\end{equation*}
   for every $\ov\alpha,\ov\alpha^*:[-\tau,0]_\Z\to\R^N$ and every $(m,n) \in \Delta$.
\end{corollary}

\begin{proof}
   If $\mathcal{M}$ is an M-matrix, then (see \cite[Theorem 5.1] {Fiedler-1986}) there is $\ov d=(d_1,\ldots,d_N)>0$ such that $\mathcal{M} \, \ov d > 0$, i.e.,
	\begin{equation} \label{eq:a_i^->alpha...2}
		d_ia_i^- > \sum_{j=1}^{N} d_jb_{ij}^+ F_j.
	\end{equation}

   The change $y_i(m)=d_i^{-1}x_i(m)$, $m\in\N_0$ and $i=1,\ldots,N$, transforms \eqref{HOPFIELD-Xu-Wu} into
	\begin{equation*} 
		y_i(m+1)=y_i(m)\e^{-a_i(m)h}+\theta_i(m)
      \left[ \dst\sum_{j=1}^N \tilde{b}_{ij}(m)\tilde{f}_j(y_j(m-\tau(m)))+\tilde{I}_i(m) \right],
	\end{equation*}
	where
	$$
	\tilde{b}_{ij}(m)=d_i^{-1}b_{ij}(m),\ \ \tilde{f}_j(u)=f_j(d_j u),\ \ \text{and} \ \ \tilde{I}_i(m)=d_i^{-1}I_i(m),
	$$
	for $m\in\N_0$ and $u\in\R$. As $f_j$ are Lipschitz functions with constant $F_j$, then $\tilde{f}_j$ are also Lipshitz functions with constant $\tilde{F}_j=d_jF_j$. From \eqref{eq:a_i^->alpha...2} we have
	$$
	a_i^- > \sum_{j=1}^{N}d_i^{-1}b_{ij}^+d_jF_j
	$$
	which is equivalent to
	$$
	a_i^- > \sum_{j=1}^{N}\tilde{b}_{ij}^+\tilde{F}_j
	$$
	and the result follows from the Corollary \ref{thm:aplicacao-cont-Xu-Wu}.
\end{proof}

Now we improve \cite[Theorem 4.2]{Xu_Wu-ADE-2013}, which proves the existence and global stability of the periodic solution of the Hopfield neural network model~\eqref{HOPFIELD-Xu-Wu} with periodic coefficients. Let $\omega\in\N$ and consider the model \eqref{HOPFIELD-Xu-Wu} where  $a_i, b_{ij}, I_i, \tau : \N_0 \to \R$ are $\omega$-periodic functions.

\begin{theorem} 
   If $\mathcal{M}$ is a $M$-matrix, then the model~\eqref{HOPFIELD-Xu-Wu} has a unique $\omega$-periodic solution which is globally exponentially stable.
\end{theorem}

\begin{proof}
   Let $n \in \N_0$. From Corollary \ref{thm:aplicacao-cont-Xu-Wu-2}, there are $\mu>0$ and $C>1$ such that
	\begin{equation} \label{eq:estabilidade-periodico}
		\|\ov x_m(\cdot,n,\ov\alpha)-\ov x_m(\cdot,n,\ov\alpha^*)\|
		\le  \,C \e^{-\mu(m-n)} \|\ov\alpha-\ov \alpha^*\|,
	\end{equation}
   for all $m\geq n$ and all $\ov\alpha,\ov\alpha^*\in X^N$. Now, choosing an integer $k \in\N$ such that
	\begin{equation} \label{eq:est-exp-cont}
		C\e^{-\mu k\omega}<1
	\end{equation}
   and defining the map $P:X^N\to X^N$ by $P(\ov\alpha)=\ov x_{n+\omega}(\cdot,n,\ov\alpha)$. For $\ov \alpha,\ov \alpha^*\in X^N$, we have
	\begin{align*}
		& \|P^k(\ov \alpha)-P^k(\ov \alpha^*)\|\\
		& = \|P(P^{k-1}(\ov \alpha))-P(P^{k-1}(\ov \alpha^*))\|\\
		& = \|\ov x_{n+\omega}(\cdot,n,P^{k-1}(\ov \alpha))
		-\ov x_{n+\omega}(\cdot,n,P^{k-1}(\ov \alpha^*))\|\\
		& = \|\ov x_{n+\omega}(\cdot,n,\ov x_{n+\omega}
         (\cdot,n,P^{k-2}(\ov \alpha)))
		-\ov x_{n+\omega}
		(\cdot,n,\ov x_{n+\omega}(\cdot,n,P^{k-2}(\ov\alpha^*)))\|,
	\end{align*}
	and, as the model  \eqref{HOPFIELD-Xu-Wu} is $\omega$-periodic, from \eqref{eq:estabilidade-periodico},
	\begin{align*}
		\|P^k(\ov \alpha)-P^k(\ov \alpha^*)\|
		& = \|\ov x_{n+2\omega}(\cdot,n,P^{k-2}(\ov \alpha))
		-\ov x_{n+2\omega}(\cdot,n,P^{k-2}(\ov \alpha^*))\|\\
		& = \|\ov x_{n+k\omega}(\cdot,n,\ov \alpha)
		-\ov x_{n+k\omega}(\cdot,n,\ov \alpha^*)\|\\
		& \leq C\e^{-\mu k\omega}\|\ov \alpha-\ov \alpha^*\|.
	\end{align*}
	From \eqref{eq:est-exp-cont}, the map $P^k$ is a contraction on $X^N$. As $X^N$ is a Banach space, we conclude that there is a unique point $\ov\varphi\in X^N$ such that $P^k(\ov \varphi)=\ov \varphi$. Noting that
	$$
	P^k(P(\ov \varphi))=P(P^k(\ov \varphi))=P(\ov \varphi),
	$$
	we have $P(\ov \varphi)=\ov \varphi$ which means $\ov x_{n+\omega}(\cdot,n,\ov \varphi)=\ov \varphi$.
	
	Finally, as $\ov x(m,n,\ov \varphi)$ is a solution of \eqref{HOPFIELD-Xu-Wu} with $a_i, b_{ij},I_i,\tau$ $\omega$-periodic functions, we know that $\ov x(m+\omega,n,\ov \varphi)$ is also a solution of  \eqref{HOPFIELD-Xu-Wu} and
	$$
	\ov x(m,n,\ov \varphi)
	=\ov x(m,n,\ov x_{n+\omega}(\cdot,n,\ov \varphi))
	=\ov x(m+\omega,n,\ov \varphi)
	$$
	for every $m \ge n$. Therefore $\ov x(m,n,\ov \varphi)$ is a $\omega$-periodic solution of \eqref{HOPFIELD-Xu-Wu} and, from \eqref{eq:estabilidade-periodico}, all other solutions converge to it with exponential rates.
\end{proof}
\section{Stability of nonuniform contractions}\label{section:stability...}
Let $Y$ be a Banach space and denote by $X$ the space of functions
   $$ \alpha:[r,0]_\Z \to Y$$
equipped with the norm $\|\alpha\| = \max\limits_{j\in[r,0]_{\Z}}|\alpha(j)|$, where $|\cdot|$ is the norm in $Y$. Given $N \in \N$, we are going to consider the cartesian products $X^N$ and $Y^N$ equipped with the supremum norm, i.e., for $\ov\alpha = \prts{\alpha_1, \ldots, \alpha_N} \in X^N$ and $\ov y = \prts{y_1, \ldots, y_N} \in Y^N$, we have
   $$ \|\ov\alpha\|
      = \max_{i=1, \ldots,N} \|\alpha_i\|
      = \max_{i=1, \ldots,N} \prts{\max_{j=r, \ldots, 0} \abs{\alpha_i(j)}}$$
and
   $$ |\ov y| = \max_{i=1, \ldots, N} |y_i|.$$

Given a function $\ov x \colon [r,+\infty[_\Z \to Y^N$ we define, for each $m \in \N_0$, $\ov x_m \in X^N$ by
   $$ \ov x_m(j) = \ov x(m+j), \ j=r, r+1, \ldots, 0.$$

Let $\ov f_m : X^N \to Y^N$ be Lipschitz perturbations such that
\begin{equation} \label{cond-f-0}
	\ov f_m(0)=0
\end{equation}
for every $m \in \N_0$.

We are going to consider the following nonlinear delay difference equation
\begin{equation} \label{eq:nonlinear}
	\ov x(m+1) = \ov L_m \ov x_m + \ov f_m(\ov x_m), \ m \in \N_0,
\end{equation}
where, for each $m \in \N_0$, $\ov L_m \colon X^N \to Y^N$ is given by
\begin{equation} \label{def:L_m}
	\ov L_m \ov\alpha
	= \prts{L^{(1)}_m \alpha_1, L^{(2)}_m \alpha_2, \ldots, L^{(N)}_m \alpha_N},
\end{equation}
with $L^{(i)}_m \colon X \to Y$ a bounded linear operator for $i=1, \ldots, N$.
For each $n \in \N_0$ and $\ov\alpha \in X^N$, we obtain a unique function
   $$ \ov x \colon [n+r,+\infty[_\Z \to Y^N,$$
denoted by $\ov x (\cdot, n, \ov\alpha)$, such that $\ov x_n = \ov \alpha$ and \eqref{eq:nonlinear} holds. Consequently, for each $(m,n)\in \Delta$, we can define the operator $\ov\cF_{m,n}:X^N \to X^N$ given by
\begin{equation*} 
	\ov\cF_{m,n}\prts{\ov\alpha}
	= \ov  x_n (\cdot,n,\ov\alpha), \ \ov\alpha \in X^N.
\end{equation*}

Associated with equation~\eqref{eq:nonlinear}, we will consider the linear difference equation
\begin{equation} \label{eq:lin:dif}
	v_i(m+1)= L^{(i)}_m(v_{i,m})
\end{equation}
$i=1, \ldots, N$, where $v_{i,m} \in X$ is defined, as usual, by $v_{i,m}(j)= v_i(m+j)$, $j=r,r+1,\ldots,0$ and $L^{(i)}_m$ is given by~\eqref{def:L_m}. For each $n \in \N_0$ and $\alpha_i \in X$, we obtain a unique function $v_i:[n+r,+\infty[_\Z \to Y$, denoted by $v_i(\cdot, n, \alpha_i)$, such that $v_{i,n}=\alpha_i$ verifying~\eqref{eq:lin:dif}.

For each $(m,n) \in \Delta$ and $i=1, \ldots, N$, we define the operator $\cA^{(i)}_{m,n}:X\to X$ by
\begin{equation*} 
	\cA^{\prts{i}}_{m,n}\alpha_i
   = v_{i,m}(\cdot,n,\alpha_i), \ \alpha_i \in X.
\end{equation*}
We can easily verify that
\begin{enumerate}[\lb=$\alph*)$,\lm=6mm]
	\item $\cA^{(i)}_{m,n}$ is linear for each $(m,n) \in \Delta $;
	\item $\cA^{(i)}_{m,m}=\Id$;
	\item $\cA^{(i)}_{l,m} \cA^{(i)}_{m,n}= \cA^{(i)}_{l,n}$ for $(l,m),(m,n) \in \Delta $.
\end{enumerate}

It is easy to prove by induction in $m$ (see \cite{Barreira_Valls-JDE-2007-238-(470-490)}) that
   $$ \ov\cF_{m,n}(\ov\alpha)
      = \prts{\cF^{(1)}_{m,n}(\ov\alpha), \ldots,
         \cF^{(N)}_{m,n}(\ov\alpha)},$$
where, for $i=1, \ldots, N$,
   $$ \cF^{(i)}_{m,n} (\ov\alpha)
      = \cA_{m,n}^{(i)} \alpha_i
         + \dsum_{k=n}^{m-1} \cA^{(i)}_{m,k+1} \Gamma f^{(i)}_k(\ov x_k),$$
$\ov \alpha = \prts{\alpha_1, \ldots, \alpha_N}$, $\ov f_k(\ov x_k) = \prts{f^{(1)}_k(\ov x_k), \ldots, f^{(N)}_k(\ov x_k)}$ and $\Gamma \colon Y  \to X$ is defined by
   $$ \begin{array}{rcl}
         \Gamma \, u \colon [r,0]_\Z & \!\!\!\to & \!\!\! Y\\
         j \ \ \ & \!\!\!\mapsto & \!\!\!\Gamma u(j) =
         \begin{cases}
            u & \text{ if } j=0,\\
            0 & \text{ if } j < 0,
         \end{cases}
      \end{array}$$
for all $u\in Y$.

\begin{theorem}\label{thm:abstract}
   Let $\ov f_m \colon X^N \to Y^N$ be Lipschitz functions such that~\eqref{cond-f-0} is satisfied and consider equation~\eqref{eq:nonlinear}. Let $\prts{a^{(i)}_{m,n}}_{(m,n) \in \Delta}$, $i=1, \ldots, N$, and let $\prts{a'_{m,n}}_{(m,n) \in \Delta}$ be double sequences such that
	  $$ \|\cA^{(i)}_{m,n}\|
	     \le a^{(i)}_{m,n}
	     \le a'_{m,n}$$	
   for all $(m,n) \in \Delta$, where $\cA^{(i)}_{m,n}$ are the evolutions operators of equation~\eqref{eq:lin:dif} derived from equation~\eqref{eq:nonlinear}.
	Assume that
	\begin{equation*} 
		\lbd
		:= \max_{i=1, \ldots, N}\prtsr{
			\sup_{\prts{m,n} \in \Delta}
			\set{\dfrac{1}{a'_{m,n}}
				\sum_{k=n}^{m-1} a^{(i)}_{m,k+1} \Lip(f^{(i)}_k) a'_{k,n}}}
		< 1.
	\end{equation*}
	Then
	\begin{equation*}
		\|\ov\cF_{m,n} (\ov\alpha)\|
		\le \dfrac{1}{1-\lbd} \, a'_{m,n} \|\ov\alpha\|
	\end{equation*}
	for every $(m,n) \in \Delta$.
\end{theorem}

\begin{proof}
   Given $n \in \N_0$ and $\ov\alpha \in X^N \setm{0}$, let $\cC_{n,\ov\alpha}$ be the space of functions
      $$ \ov x \colon [n+r,+\infty[_\Z \to Y^N$$
   such that
	\begin{align*}
		& \ov x_n = \ov\alpha\\
		& \abs{\ov x}_{\cC_{n,\ov\alpha}}
		= \sups{\dfrac{\|\ov x_m\|}{a'_{m,n} \|\ov\alpha\|}
         \colon m \ge n} <+\infty.
	\end{align*}
   It is clear that $\cC_{n,\ov\alpha}$ is a complete metric space with the metric defined by
	  $$ d(\ov x,\ov y)
	     = \abs{\ov x-\ov y}_{\cC_{n,\ov\alpha}}
	     = \sups{\dfrac{\|\ov x_m- \ov y_m\|}{a'_{m,n} \, \|\ov\alpha\|}
		    \colon m \ge n}.$$
	
	For every $\ov x \in \cC_{n,\ov\alpha}$ we define
	  $$ (J\ov x)_m =
	     \begin{cases}
	        \ov\alpha & \text{ if } m=n\\
	        (\xi^{(1)}_m,\ldots,\xi^{(N)}_m)
         & \text{ if } m > n.
	\end{cases}$$
	where, for each $i=1,\ldots,N$,
	  $$ \xi^{(i)}_m
         = \cA^{(i)}_{m,n} \alpha_i
            + \dsum_{k=n}^{m-1} \cA_{m,k+1}^{(i)} \Gamma f_k^{(i)}(\ov x_k).$$
	Since for $m > n$ we have
	\begin{align*}
		\|\prts{J\ov x}_m\|
		& \le \max_{i=1,\ldots,N} \set{\|\cA^{(i)}_{m,n} \alpha_i\|
         + \dsum_{k=n}^{m-1} \|\cA^{(i)}_{m,k+1} \|
            \|\Gamma f^{(i)}_k(\ov x_k)\|}\\
		& \le \max_{i=1,\ldots,N} \set{ a^{(i)}_{m,n} \|\alpha_i\|
         + \dsum_{k=n}^{m-1} a^{(i)}_{m,k+1} \Lip(f^{(i)}_k) \|\ov x_k\|}\\
		& \le \max_{i=1,\ldots,N} \set{ a'_{m,n} \|\alpha_i\|
			+ \dsum_{k=n}^{m-1} a^{(i)}_{m,k+1} \Lip(f^{(i)}_k) a'_{k,n}
         \|\ov\alpha\| \abs{\ov x}_{\cC_{n,\ov\alpha}}}\\
		& \le \prts{1 + \lbd  \abs{\ov x}_{\cC_{n,\ov\alpha}}}
         a'_{m,n} \|\ov\alpha\|
	\end{align*}
	and this implies that $J\ov x$ belongs to $\cC_{n,\ov\alpha}$ and
	\begin{equation} \label{eq:|Jx|<=1+lbd.|x|}
		\abs{J \ov x}_{\cC_{n,\ov\alpha}}
		\le 1 + \lbd  \abs{\ov x}_{\cC_{n,\ov\alpha}}.
	\end{equation}
	Hence $J \colon \cC_{n,\ov\alpha} \to \cC_{n,\ov\alpha}$.
	
   Now we prove that $J$ is a contraction. For every $\ov x, \ov y \in \cC_{n,\ov\alpha}$  and every $m > n$ we have
	\begin{align*}
		\|\prts{J\ov x-J\ov y}_m\|
		& = \|\prts{J\ov x}_m-\prts{J\ov y}_m\|\\
		& \le \max_{i=1,\ldots,N} \set{\dsum_{k=n}^{m-1} \|\cA^{(i)}_{m,k+1} \|
         \|\Gamma f^{(i)}_k(\ov x_k) - \Gamma f^{(i)}_k(\ov y_k)\|}\\
		& \le \max_{i=1,\ldots,N}  \set{\dsum_{k=n}^{m-1} a^{(i)}_{m,k+1}
         \Lip(f^{(i)}_k) \|\ov x_k - \ov y_k\|}\\
		& \le \max_{i=1,\ldots,N} \set{\dsum_{k=n}^{m-1} a^{(i)}_{m,k+1}
         \Lip(f^{(i)}_k) a'_{k,n} \|\ov\alpha\|
         \abs{\ov x- \ov y}_{\cC_{n,\ov\alpha}}}\\
		& \le a'_{m,n} \|\ov\alpha\| \lbd \abs{\ov x- \ov y}_{\cC_{n,\ov\alpha}}
	\end{align*}
	and thus
	  $$ \abs{J\ov x - J\ov y}_{\cC_{n,\ov\alpha}}
	     \le \lbd \abs{\ov x-\ov y}_{\cC_{n,\ov\alpha}}.$$
   Since $\lbd < 1$, $J$ is a contraction and by the Banach fixed point theorem $J$ has a unique fixed point $\ov x^*$. By~\eqref{eq:|Jx|<=1+lbd.|x|} it follows that the fixed point $\ov x^*$ verifies
	  $$ \abs{\ov x^*}_{\cC_{n,\ov\alpha}}
	     \le \dfrac{1}{1-\lbd}$$
	and this proves the theorem.
\end{proof}
\section*{Acknowledgments}
Ant\'onio J. G. Bento and C\'esar M. Silva were partially supported by FCT through CMA-UBI (project PEst-OE/MAT/UI0212/2014).

Jos\'e J. Oliveira was supported by the Research Centre of Mathematics of the University of Minho with the Portuguese Funds from the “Funda\c c\~ao para a Ci\^encia e a Tecnologia”, through the Project PEstOE/MAT/UI0013/2014.

\end{document}